\documentclass[]{article}
\usepackage{amssymb}
\usepackage{amsmath}
\usepackage{amsthm}
\usepackage{color}
\usepackage[
top    = 3.00cm,
bottom = 3.00cm,
left   = 3.00cm,
right  = 2.80cm]{geometry}
\usepackage[style=alphabetic,natbib=true,bibencoding=ascii,backend=bibtex]{biblatex}
\usepackage{hyperref}
\usepackage{etoolbox}
\usepackage{tikz-cd}
\usepackage{faktor}
\usepackage{xfrac} 
\apptocmd{\thebibliography}{\raggedright}{}{}

\usepackage{titletoc}

\dottedcontents{section}[1.5em]{\bfseries}{1.3em}{.6em}

\newtheorem{theorem}{Theorem}[section]
\newtheorem{lemma}[theorem]{Lemma}
\newtheorem{proposition}[theorem]{Proposition}

\newtheorem{remark}{Remark}[section]
\newtheorem{definition}{Definition}[section]
\title{On the level raising of cuspidal eigenforms modulo prime powers}
\author{Emiliano Torti \\ University of Luxembourg\\ emiliano.torti@uni.lu or torti.emiliano@gmail.com}

\begin{document}

\maketitle

\begin{abstract}
In this article we prove level raising for cuspidal eigenforms modulo prime powers (for odd primes) of weight $k\geq 2$ and arbitrary character, extending the result in weight two established by the work of Tsaknias and Wiese and generalizing (partially) Diamond-Ribet's celebrated level raising theorems.
\end{abstract}

\section{Introduction}
The problem of raising the level of newforms modulo some power of a prime $p$ is part of the study of congruences between modular forms of different levels. By congruences between modular forms we mean congruences between almost all coefficients (i.e. except from a finite number) given by the $q$-expansion principle. \\ 
The level raising phenomenon (modulo a prime $p$) was extensively studied in the past thirty years and it was definitely understood in the classical case thanks to the work of Ribet (see \cite{Ribet1990}) in weight two and trivial character, and Diamond (see \cite{Diamond1991}) for weight $k\geq 2$ and general characters.\\
In particular, they proved the following:

\begin{theorem} (Ribet, Diamond)\\
Let $f$ be a newform of weight $k\geq 2$, level $N$, character $\chi$ and let $p$ be a prime not dividing $N$. If $l$ is a prime not dividing $pN$ and $p\nmid \frac{1}{2}\varphi(N)Nl(k-2)!$,
then the following are equivalent:
$$
\begin{aligned}
&\text{(a) there exists a } l\text{-newform} \; \text{of level }lN, \text{ say } g, \text{ such that } f\equiv g \mod \mathfrak{p}\\
&\text{(b)}\; a_l ^2 \equiv \chi(l)l^{k-2}(l+1)^2 \mod\mathfrak{p}.
\end{aligned}
$$
Here, the symbol $\mathfrak{p}$ denotes a prime ideal of the coefficient field of $f$ lying above the rational prime $p$, and $a_l$ denote the $l$-th coefficient of the q-expansion of $f$. 
\end{theorem}
The fundamental idea of the existence of a non-trivial congruence module introduced and developed by Ribet in a geometric context (considering Jacobians attached to modular curves of the form $X_0(N)$) was refined in cohomological terms by Diamond (see \cite{Diamond1989}). Proving the existence of a congruence module, whose non-triviality is granted by the level raising condition (i.e. the condition (b) in the above theorem), is the heart of both the proofs of Ribet and Diamond, and its cohomological construction will allow us to apply Ribet's analysis in a slightly more general context.\\
As a natural extension of studying level raising modulo $p$, one could ask if the same holds modulo $p^n$ for some positive integer $n$ and if the natural generalization of the level raising level condition is the suitable one in the prime powers setting. More specifically, given a newform of level $N$ whose $l$-th coefficient (for a prime $l$ not dividing $pN$) satisfies a certain level raising condition modulo $p^n$, one could ask if it is it true that there exists a newform (at $l$) of level $lN$ which is congruent modulo $p^n$ to the original newform $f$ of level $N$. We will refer to this question as full level raising problem.\\
It turns out that for classical newforms satisfying the Ribet-Diamond's level raising condition modulo prime powers, this is not always the case and we will present some counterexample in the last section. Nevertheless, there is a result that partially answers the question, namely Diamond proved (see Thm. 2, \cite{Diamond1991}) that if the level raising condition holds modulo some power of $p$, then this gives rise to a family of congruences modulo lower powers of $p$ between the original newform and possibly different newforms (new at $l$) of level $lN$. We will state the theorem in precise terms in the last section (see Thm. \ref{Diamond}). \\
This leads to ask if it is possible to prove the level raising modulo some prime power if we weaken the definition of a reduced cusp form modulo some prime power. Indeed, if $p$ is a prime and $\mathbb{K}$ is a finite extension of $\mathbb{Q}_p$ with ring of integers $\mathcal{O}_\mathbb{K}$, there is a natural arithmetic definition of cuspidal eigenform with coefficients in a complete Noetherian local $\mathcal{O}_\mathbb{K}$-algebra with finite residue characteristic, say $A$. 

\begin{definition}
\label{cusp}
Let $k\geq 2$ and $N\geq 5$ be positive integers. A cuspidal eigenform of weight $k$, level $N$ (coprime with $p$) and coefficients in $A$ is an element of the set $\text{Hom}_{\mathcal{O}_\mathbb{K} \text{-Alg}}(\mathbb{T}, A)$, i.e. it is an $\mathcal{O}_\mathbb{K}$-algebra homomorphism from a certain Hecke algebra $\mathbb{T}$ defined over $\mathcal{O}_\mathbb{K}$ acting on the space of classical cusp forms of weight $k$ and level $\Gamma_1(N)$ to $A$.
\end{definition}
This definition goes back to Carayol (see \cite{Carayol1994}) and it strictly depends on the fixed coefficient ring $A$. A motivation for this definition comes from the classical case (see sec.2.1 in \cite{Carayol1994}). Indeed, assuming $p\nmid N$, there exists a perfect pairing of $\mathcal{O}_\mathbb{K}$-modules between $S_k (\Gamma_1 (N),\mathcal{O}_\mathbb{K} )$ and the Hecke $\mathcal{O}_\mathbb{K}$-algebra $\mathbb{T}_N$ given by 
$\varphi: \mathbb{T}_N \times S_k (\Gamma_1 (N), \mathcal{O}_\mathbb{K} ) \rightarrow \mathcal{O}_\mathbb{K}$ 
where $\varphi(T,f)=a_1 (Tf)$ (the first coefficient in the $q$-expansion of $Tf$).\\
Now, by a cuspidal  eigenform modulo some prime power we simply mean a cuspidal eigenform defined as above with coefficients in the $\mathcal{O}_\mathbb{K}$-algebra $\mathcal{O}_\mathbb{K}/(\pi^r)$ where $\pi$ is a uniformizer and $r$ is a positive integer. A generalized definition but independent of the chosen ring of coefficients is given by Chen, Kiming and Wiese (see \cite{CKW13}). We will present a natural notion of being new for a cuspidal eigenform modulo prime powers later in the article.\\
A cusp eigenform modulo some prime power does not lift in general to an eigenform in characteristic zero once we fix the level and the weight, as it is observed by Calegari and Emerton (see \cite{CaEm04}). More explicitly, Chen, Kiming and Wiese presented a general construction for non-liftable cuspidal eigenforms modulo $p^2$ and an explicit example for $p=3$ (see sec. 5.3, \cite{CKW13}). In general, there is no known characterization which determines whether or not a cuspidal eigenforms modulo some prime power comes from the reduction of an eigenform in characteristic zero. The only known case is the one of cuspidal eigenforms modulo $p$, thanks to the Deligne-Serre lifting lemma when the character is trivial and $k\geq 2$ (see Lemma 6.11, \cite{DelSer74}). If the character $\chi$ is not trivial, thanks to a lemma of Carayol it is still possible to lift in characteristic 0 a cuspidal form modulo a prime under the extra condition $p\geq 5$ (see Prop. 1.10 \cite{Edi95})\\
Coming back to the level raising problem for cuspidal eigenforms modulo prime powers, the case of weight two and trivial character was studied by several authors (see sec. 5.6, \cite{BD05}, and sec. 1.4, \cite{BBV16}). In particular, the best known result has been proved by Tsaknias and Wiese with different techniques than the previously mentioned articles (see Thm. 5 in \cite{TsakniasWiese2017}). 
We will extend that result proving the following:

\begin{theorem}
	\label{MainTheorem}
	Let $f:\mathbb{T}_N\rightarrow \mathcal{O}_\mathbb{K}/(\pi^r)$ be a cusp eigenform of level $N$, weight $k$ and character $\chi$. Assume that its associated residual Galois representation $\bar{\rho}_{f}$ is absolutely irreducible. Suppose that $p$ does not divide $\varphi(N)N(k-2)!$ and the field $\mathbb{K}$ is sufficiently big.\\
If $k=p$ or $k=p+1$, assume that the localized Hecke algebra $\mathbb{T}_\mathcal{M}$ is Gorenstein, where $\mathcal{M}$ is the kernel of the reduction of $f$ modulo $\pi$.\\
	Let $l$ be a prime which does not divide $pN$. Then the two following statements are equivalent:
	
	$$
	\begin{aligned}
	(i)&\;\;T_l-\epsilon(l+1)R_l\in\text{Ker}(f) \text{ for some } \epsilon\in \{\pm1\}, \text{ where }R_l\in\mathbb{T}_N \text{ and }R_l ^2=l^{k-2}\chi(l)\\
	(ii)&\;\text{ there exists a cusp eigenform } g:\mathbb{T}_{N;l}\rightarrow \mathcal{O}_\mathbb{K} \text{ of weight } k \text{ and character } \chi \text{ such that: }
	\end{aligned}
	$$
	$$
	\begin{aligned}
	(a)&\;\; f(T_q)=g(T_q) \;\;\text{for all primes }\; q\nmid lpN,\\
	(b)&\;\; g\text{ is new at }l. \\
	\end{aligned}
	$$
	Here, the symbol $\mathbb{T}_{N;l}$ denotes the Hecke $\mathcal{O}_\mathbb{K}$-algebra acting on the space of classical cusp forms of weight $k$ and level  $\Gamma_1(N)\cap\Gamma_0(l)$.
\end{theorem}
In section 2.1, we will see how to associate to each cuspidal eigenform modulo $\pi^r$ a Galois representation with coefficients in $\mathcal{O}_\mathbb{K} /(\pi^r)$. This was done by Carayol  (see Thm. 3, \cite{Carayol1994}) using deformation theory. Carayol's construction will allow us to transpose properties of classical newforms to properly defined newforms modulo prime powers and this will lead us to determine the ``correct'' level raising condition for the level raising problem in this setting. \\
In section 2.2, we will study such condition, comparing it with the one considered by Ribet and Diamond in the classical case. Finally, we will show the necessity of this condition if we assume the existence of a congruence coming from level raising. \\
In section 2.3, we state the main result of this article and in section 2.4 we give a proof. In section 3, we present a family of examples of general interest for the level raising phenomenon.
\section{Raising the level of cusp eigenforms modulo prime powers of weight $k\geq 2$}
In order to generalize Tsaknias-Wiese's level raising result (see \cite{TsakniasWiese2017}) to cuspidal eigenforms modulo prime powers of weights greater than 2, we first need to establish a necessary condition for the level raising. We will then compare it with the one considered in the classical case by Ribet (see \cite{Ribet1990}, see also \cite{Diamond1989}) and Diamond (see \cite{Diamond1991}). \\
As in the previous sections, $\mathbb{K}$ will denote a finite extension of $\mathbb{Q}_p$ for a fixed prime $p\geq 5$, $\mathcal{O}_{\mathbb{K}}$ will denote its ring of integers and $\pi$ will denote a uniformizer. We will assume that $p$ is different from 2.\\
Let $N\geq 5$ and $k\geq 2$ two positive integers. Throughout the paper, we consider them fixed. From now on, we assume that the prime $p$ does not divide $\varphi(N)$, where $\varphi$ is the Euler's totient function.\\
We fix once and for all a Dirichlet character $\chi: (\mathbb{Z}/N\mathbb{Z})^{\times}\rightarrow \mathcal{O}^{\times}_\mathbb{K}$. We will denote by $\mathbb{T}_N$ the $\mathcal{O}_\mathbb{K}$-subalgebra of $\mathrm{End}_{\mathbb{K}}(S_k(\Gamma_1(N),\chi,\mathbb{K}))$ generated by the Hecke operators $T_n$ for $n\geq 1$. If $m\in\mathbb{N}$ is coprime with $pN$, the $\mathcal{O}_\mathbb{K}$-algebra $\mathbb{T}_N$ contains the modified diamond operators $S_m:=m^{k-2} \langle m\rangle=m^{k-2}\chi(m)$; they are scalar operators.\\
As we will see in details later on, the condition that $p$ does not divide $\varphi(N)$ ensure us that $S_k(\Gamma_1(N),\chi,\mathcal{O}_\mathbb{K})$ is a direct summand in $S_k(\Gamma_1(N),\mathcal{O}_\mathbb{K})$.\\
Similarly, if $l$ is a prime which does not divide $pN$, we define the Hecke algebra $\mathbb{T}_{N;l}$ of level $lN$ as the $\mathcal{O}_\mathbb{K}$-subalgebra of $\mathrm{End}_\mathbb{K}(S_k(\Gamma_1(N)\cap \Gamma_0(l),\chi ,\mathbb{K}))$ generated by the Hecke operators. Note that $S_l\in\mathbb{T}_{N;l}$ since we are considering the congruence subgroup $\Gamma_1(N)\cap \Gamma_0(l)$.

\subsection{Galois representations attached to cuspidal eigenforms modulo prime powers}
Let $A$ be a complete Noetherian local $\mathcal{O}_{\mathbb{K}}$-algebra with finite residue field of characteristic $p$. 
Let $f$ be an element in the set $\mathrm{Hom}_{\mathcal{O}_\mathbb{K}\text{-alg}}(\mathbb{T}_N,A)$, or in other words, a cuspidal eigenform with coefficients in $A$ and character $\chi$.\\
First of all, we will start to associate to $f$ a residual Galois representation with coefficients in characteristic $p$.\\ 
Let $\bar{f}$ be the reduction of $f$ modulo the unique maximal ideal $\mathcal{M}_A$ of $A$, i.e. the composition of $f$ with the natural projection map from $A$ to the quotient $A/\mathcal{M}_A$. Now, $\bar{f}$ is a cusp eigenform whose coefficients lie in a finite extension of $\mathbb{F}_p$. \\
By a lemma of Carayol (see Prop. 1.10 \cite{Edi95}) and recalling the assumption $p\geq 5$, we obtain that $\bar{f}$ comes from the reduction of a classical (i.e. holomorphic) Hecke eigenform $g\in S_k(\Gamma_1(N),\chi)$ with coefficients in an order of a number field.\\
By a theorem of Deligne and Shimura, we can associate to $g$, and so to $\bar{f}$, a semisimple residual Galois representation 
$\rho: \mathrm{Gal}(\bar{\mathbb{Q}}/\mathbb{Q})\rightarrow \mathrm{Gl}_2(\bar{\mathbb{F}}_p)$ which is unramified outside the primes which divide $pN$ and such that $\mathrm{Trace}(\rho(\mathrm{Frob}_q))=\bar{f}(T_q)$ and $\mathrm{Det}(\rho(\mathrm{Frob}_q))=q\bar{f}(S_q)=q^{k-1}\chi(q)$ for any prime $q$ which does not divide $pN$. We will denote this residual representation associated to $\bar{f}$ by $\rho_{\bar{f}}$.\\
Assuming that $\rho_{\bar{f}}$ is absolutely irreducible, a theorem of Carayol (Thm. 3, \cite{Carayol1994}), allows us to associate to each cusp eigenform with coefficients in $A$, a Galois representations $\rho_{f}$ with coefficient in $A$ whose reduction modulo the maximal ideal $\mathcal{M}_A$ coincides with $\rho_{\bar{f}}$. In other words, the representation $\rho_{f}$ is a deformation of $\rho_{\bar{f}}$. \\
In other terms, the following result holds:

\begin{theorem} (Carayol)\\
	Let $f: \mathbb{T}_N \rightarrow A$ be a cusp eigenform of level $N$, weight $k$, character $\chi$ and coefficients in $A$ such that its residual associated Galois representation is absolutely irreducible. Then there exists a unique (up to isomorphism) continuous representation:
	$$\rho_f : \mbox{Gal}(\bar{\mathbb{Q}}/\mathbb{Q}) \longrightarrow \mbox{Gl}_2(A)$$
	which is unramified at primes not dividing $pN$, and which satisfies the relations:
	$$
	\begin{cases}
	\mbox{Trace}(\rho_f(\mbox{Frob}_q))= f(T_q) \\
	\mbox{Det}(\rho_f(\mbox{Frob}_q))= f(qS_q)=qf(S_q)=q^{k-1}\chi(q) \\
	\end{cases}
	\;\;\;\; \mbox{for all primes }q\nmid pN.
	$$
\end{theorem}
Hence, let $r\geq 1$ be an integer and let $f$ be a cusp eigenform on $\Gamma_1(N)$ of character $\chi$, weight $k\in\mathbb{N}_{\geq2}$, level $N\in\mathbb{N}$ and coefficients in $A:=\mathcal{O}_{\mathbb{K}}/{(\pi^r)}$.\\
Using Carayol's theorem, we associated to $f$ a Galois representation
$$\rho_f : \mbox{Gal}(\bar{\mathbb{Q}}/\mathbb{Q}) \longrightarrow \mbox{Gl}_2(\mathcal{O}_{\mathbb{K}}/{(\pi^r)})$$
such that:
$$
\begin{cases}
\mbox{Trace}(\rho_f(\mbox{Frob}_q))= f(T_q) \\
\mbox{Det}(\rho_f(\mbox{Frob}_q))= f(qS_q)=qf(S_q)=q^{k-1}\chi(q)\\
\end{cases}
\;\;\;\; \mbox{for all primes }q\nmid pN
$$
and which is unramified outside the set of primes dividing $pN$.\\
Now, since we are interested in congruences between cusp eigenforms of level $N$ and level $lN$ modulo some power of $p$ , we need to understand what happens to these representations at the (possibly bad) prime l.\\
The strategy will be the following: first, we will briefly recall the construction of the universal representation in Carayol's proof of the above theorem, then we will consider a cusp eigenform modulo prime powers $g$ of level $lN$ (i.e. with respect to the congruence subgroup $\Gamma_1(N)\cap\Gamma_0(l)$ where $l$ does not divide $N$) which is new at $l$, and we will use the Carayol's universal representation construction to deduce some properties on the Galois representation $\rho_g$ associated to $g$.\\
As long as it will be possible, we will work in the general context where coefficients of the cusp eigenforms lie in a complete Noetherian local $\mathcal{O}_{\mathbb{K}}$-algebra $A$ with finite residue field of characteristic $p$.\\
Let $f$ be a cusp eigenform defined as in the above theorem, then we define $\bar{f}$ the reduction of $f$ modulo the maximal ideal $\mathcal{M}_A$ of $A$. We denote by $\mathcal{M}$ the kernel of $\bar{f}: \mathbb{T}_N\rightarrow A/\mathcal{M}_A \hookrightarrow \bar{\mathbb{F}}_p$; it is a maximal ideal.\\
Since $A$ is complete, by the universal property of completion, we have that $f$ factors through the completion of the Hecke algebra at $\mathcal{M}$, i.e. we have the following commutative diagram:
$$
\begin{tikzcd}
	& \mathbb{T}_\mathcal{M}  \arrow{d}{\tilde{f}} \\
	\mathbb{T} \arrow{r}[swap]{f} \arrow{ur}{\lambda^{\text{univ}}} & A,    
\end{tikzcd}
$$
where $\lambda^{\text{univ}}$ is the natural ring homomorphism from $\mathbb{T}:=\mathbb{T}_N$ to its completion at $\mathcal{M}$. Note that, when it will be possible (i.e. clear from the context), we will drop the subscript denoting the level considered in order to simplify the notation.\\
This implies that if we are able to construct a Galois representation of the so-called (e.g. \cite{Carayol1994}) universal Hecke eigenform $\lambda^{\text{univ}}$, then we can associate to $f$ the Galois representation given by the composition between the universal one and the homomorphism $\mathrm{Gl}_2(\mathbb{T}_\mathcal{M})\rightarrow \mathrm{Gl}_2(A)$ induced by $\tilde{f}$.\\
Following Carayol, we proceed as follows: first we consider the $\mathcal{O}_{\mathbb{K}}$-subalgebra of $\mathbb{T}_N$ generated by all the Hecke operators $T_n$ with gcd$(n,pN)=1$ and we denote it by $\mathbb{T}'\subset \mathbb{T}$. Note that in this context, the diamond operators are just scalars. We denote by $\hat{\mathbb{T}}'_\mathcal{M}$ the integral closure of $\mathbb{T}'_\mathcal{M}$ in $\mathbb{T}'_\mathcal{M} \otimes \mathbb{K}$, where $\mathbb{T}'_\mathcal{M}$ is the image of $\mathbb{T}'$ in $\mathbb{T}_\mathcal{M}$ (we drop the subscript that denotes the level $N$ of the Hecke algebra). Since the Hecke operators $T_n$ with gcd$(n,N)=1$ are simultaneously diagonalizable, the ring $\hat{\mathbb{T}}'_\mathcal{M}$ is isomorphic to the finite product of some rings of integers of finite extensions of $\mathbb{K}$ which we denote by $\mathbb{E}_j$ where the index $j$ runs over a finite subset $J$ of the positive integers. The finite set $J$ is in bijection with the set of classical normalized eigenforms, up to Gal$(\bar{\mathbb{K}}/\mathbb{K})$-conjugacy, whose residual Galois representation is isomorphic to $\bar{\rho}_f$. Hence, we have an injection:

$$\mathbb{T}'_\mathcal{M} \hookrightarrow \hat{\mathbb{T}}'_\mathcal{M} \cong \prod_{j\in J} \mathcal{O}_{\mathbb{E}_j}.$$
Now, fix an $i\in J$, we can consider the composition with the standard projection on the $i$-th component $\mbox{Pr}_i$ and so we get the ring homomorphism $g_i$:

$$g_i: \mathbb{T}'\rightarrow \mathbb{T}'_\mathcal{M} \hookrightarrow \hat{\mathbb{T}}'_\mathcal{M} \cong \prod_{j\in J} \mathcal{O}_{\mathbb{E}_j} \xrightarrow{Pr_i} \mathcal{O}_{\mathbb{E}_i}.$$
The above homomorphism $g_i$ can be extended to a homomorphism $g_i: \mathbb{T} \rightarrow \tilde{\mathbb{E}}_i$ where $\tilde{\mathbb{E}}_i$ is a finite extension of $\mathbb{E}_i$. In order to do so, we need to define $g_i$ on the operators $T_q$ where $q$ is a prime dividing the level $N$. Since all these operators satisfy their characteristic polynomials, it is enough to define $\tilde{\mathbb{E}}_i$ as the smallest field in which these characteristic polynomials split completely. Moreover since we have a finite number of these polynomials, we get a finite degree extension of $\mathbb{E}_i$. For simplicity, we keep denoting these extensions by $\mathbb{E}_i$.\\
Now, the ring homomorphism $g_i$ is a classical normalized eigenform with coefficients in a finite extension of $\mathbb{Q}_p$ whose residual Galois representation is isomorphic to $\bar{\rho}_f$. By a theorem of Deligne and Shimura, we can associate to each of the $g_i$ a Galois representation $\rho_{g_i}: \mbox{Gal}(\bar{\mathbb{Q}}/\mathbb{Q}) \rightarrow \mbox{Gl}_2(\mathbb{E}_i)$ such that the traces of the images of Frobenius elements at primes different from $p$ and not dividing the level are the coefficients of the $q$-expansion associated to the cusp form $g_i$. Finally, it is enough to define the representation of the universal modular form as the product of the $\rho_{g_i}$ for all the finite indexes $i$, so we get:

$$\rho^{\text{univ}}:=\prod_{i\in J} \rho_{g_i}: \mbox{Gal}(\bar{\mathbb{Q}}/\mathbb{Q})\rightarrow  \mbox{Gl}_2(\hat{\mathbb{T}}'_\mathcal{M})\cong \mbox{Gl}_2\Big(\prod_{i\in J} \mathcal{O}_{\mathbb{E}_i}\Big)\cong \prod_{i\in J} \mbox{Gl}_2(\mathcal{O}_{\mathbb{E}_i}).
$$
Now, a priori we know that the image of the universal representation is contained in $\mbox{Gl}_2(\hat{\mathbb{T}}'_\mathcal{M})$, but it actually lies inside $\mbox{Gl}_2(\mathbb{T}'_\mathcal{M})$. In order to show this, it is sufficient to observe that it follows from a theorem of Carayol (see Thm.2 in \cite{Carayol1994}; or sec. 6 in \cite{Maz97}), Chebotarev's density theorem and from the fact that the traces at Frobenius elements lie in $\mathbb{T}'_\mathcal{M}$. The existence of the representation $\rho^{\text{univ}}$ allows us to define $\rho_f$ as the Galois representation attached to $f$ via the following commutative diagram:
$$
\begin{tikzcd}
& \;\;\;\;\;\;\;\;\;\;\;\;\;\;\;\;\;\;\;\;\;\;\;\mbox{Gl}_2(\mathbb{T}'_\mathcal{M})\hookrightarrow \prod_{i\in J} \mbox{Gl}_2(\mathcal{O}_{\mathbb{E}_i})  \arrow{d}{\tilde{f}_{\star}} \\
\mbox{Gal}(\bar{\mathbb{Q}}/\mathbb{Q}) \arrow{r}[swap]{\rho_{f}:=\tilde{f}_{\star}\circ\rho^{\text{univ}}} \arrow{ur}{\rho^{\text{univ}}} & \mbox{Gl}_2(A)    
\end{tikzcd}
$$
where the homomorphism $\tilde{f}_{\star}$ is the one induced functorially by $\tilde{f}: \mathbb{T}_\mathcal{M}\rightarrow A$ on the general linear groups.\\
Now, we want to introduce a notion of newform for cuspidal eigenforms with coefficients in a complete Noetherian local $\mathcal{O}_\mathbb{K}$-algebra $A$, and deduce useful properties of its associated Galois representations. This will allow us to find a necessary congruence condition for the level raising problem in the next section.\\
Let $l$ be a prime not dividing $pN$ and let $\mathbb{T}_{N,l}^{l\text{-new}}$ be the $\mathcal{O}_\mathbb{K}$-algebra generated by all Hecke operators inside the endomorphism ring of the $\mathbb{K}$-vector space $S_k(\Gamma_1(N)\cap\Gamma_0(l),\chi,\mathbb{K})^{l\text{-new}}$. It is a natural quotient of the Hecke algebra $\mathbb{T}_{N,l}$. Note that when it will be possible (i.e. clear from the context) we will drop the subscript that denotes the level in the symbol $\mathbb{T}_{N,l}^{l\text{-new}}$.
We give the following:
\begin{definition}
Let $g: \mathbb{T}_{N;l}\rightarrow A$ be a cuspidal eigenform of level $lN$ weight $k$, character $\chi$. The cuspidal eigenform $g$ is $l$-new if it factors (as $\mathcal{O}_\mathbb{K}$-algebra homomorphism)
via the quotient $\mathbb{T}_{N,l}^{l\text{-new}}$.
\end{definition}
A careful analysis of Carayol's construction of the universal Galois representation will allow us to transpose properties of classical $l$-newforms to cuspidal eigenforms which are $l$-new in the above sense.\\
Fix a cuspidal eigenform $g:\mathbb{T}_{N,l}\rightarrow A$ of weight $k$, character $\chi$ which is new at $l$. Let $\mathcal{N}$ be the kernel of the reduction of $g$ modulo the maximal ideal of $A$. Localizing $\mathbb{T}_{N,l}$ at the maximal ideal $\mathcal{N}$ and denoting by $\mathcal{N}^{l\text{-new}}$ the image of the maximal ideal $\mathcal{N}$ via the natural projection $\mathbb{T}_{N,l}\twoheadrightarrow \mathbb{T}_{N;l}^{l\text{-new}}$, we have the following commutative diagram (note that $\mathcal{N}^{l\text{-new}}$ is still a maximal ideal):
$$
\begin{tikzcd}
& \mathbb{T}^{l\text{-new}}_{\mathcal{N}^{l\text{-new}}}  \arrow{d}{\tilde{g}} \\
\mathbb{T}^{l\text{-new}} \arrow{r}[swap]{g} \arrow{ur}{\lambda^{\text{univ}}_{l\text{-new}}} & A    
	\end{tikzcd}
	$$
where $\lambda^{\text{univ}}_{l\text{-new}}$ is the natural homomorphism from $\mathbb{T}^{l\text{-new}}$ to its localization at $\mathcal{N}^{l\text{-new}}$.\\
Hence we can consider the ``universal" Galois representation given by the product of all the Galois representations associated to the classical Hecke eigenforms which are new at $l$:
$$\rho^{\text{univ}}_{l\text{-new}}: \mbox{Gal}(\bar{\mathbb{Q}}/\mathbb{Q})\longrightarrow \mbox{Gl}_2(\mathbb{T}^{l\text{-new}}_{\mathcal{N}^{l\text{-new}}})\hookrightarrow \mbox{Gl}_2 \Big(\prod_{i\in \tilde{J}} \mathcal{O}_{\mathbb{E}_i}\Big) \cong\prod_{i\in \tilde{J}} \mbox{Gl}_2(\mathcal{O}_{\mathbb{E}_i})$$
where $\tilde{J}$ is a finite set in bijection with the classical $l$-newforms of level $N$ whose residual Galois representation is isomorphic to the residual Galois representation of $g$.\\
Now, recalling that we are always assuming that $l$ does not divide $pN$, the representation $\rho^{\text{univ}}_{l\text{-new}}$ has an explicit description at the prime $l$:

\begin{lemma}
\label{LemmaTrace}
Consider the universal Galois representation $$\rho^{\text{univ}}_{l\text{-new}}: \mathrm{Gal}(\bar{\mathbb{Q}}/\mathbb{Q})\rightarrow \mathrm{Gl}_2 (\mathbb{T}^{l\text{-new}}_{\mathcal{N}^{l\text{-new}}})$$ associated to the newforms at $l$ constructed above. Then the trace map is well defined at $\text{Frob}_l$ and it satisfies:
$$\mbox{Trace}(\rho^{univ}_{l\text{-new}})(\mbox{Frob}_l)=(l+1)U_l$$
where $U_l\in\mathbb{T}^{l\text{-new}}_{\mathcal{N}^{l\text{-new}}}.$
Equivalently, since by definition the finite set $\tilde{J}$ defined above is in bijection with the (finite) set of classical Hecke newforms at $l$ of level $lN$ (denoted by $h_i$ for $i$ in $\tilde{J}$), we have that:
$$\mbox{Trace}(\rho^{\text{univ}}_{l\text{-new}})(\text{Frob}_l)=((l+1)h_i(U_l))_{i\in \tilde{J}}\in\prod_{i\in \tilde{J}} \mathcal{O}_{\mathbb{E}_i}.$$
\end{lemma}

\begin{proof}

Let $h$ be a classical Hecke newform at $l$ of level $lN$, and consider the Galois representation $\rho_h$ attached to $h$ by Deligne and Shimura. Since $h$ is new at $l$, it is well known (e.g sec. 3.3, \cite{Wes05}) that $\rho_h$ is ordinary at the prime $l$. We should recall that being ordinary is a deformation condition (sec. 30, \cite{Maz97}). Moreover, we have a very explicit description of $\rho_h$ when restricted to the decomposition group $G_{\mathbb{Q}_l}$ (see Lemma 2.6.1, \cite{EmPoWe06}). 
Indeed, let $V_{I_l}$ be the largest quotient of the Galois module $V$ given by the representation $\rho_h$, on which the inertia group $I_l$ acts trivially. Then there exists an unramified character $\eta: G_{\mathbb{Q}_l}\rightarrow \bar{\mathbb{Q}}_p$ such that:

$$
\begin{aligned}
&(1)\;\eta(\mathrm{Frob}_l)=a_l(h)=h(U_l) \\
&(2)\; G_{\mathbb{Q}_l} \mbox{ acts on } V_{I_l} \mbox{ via the character } \eta \\
&(3)\; {\rho_h}_{\big|_{G_{\mathbb{Q}_l}}}=
\begin{pmatrix}
\omega_p \eta & \star \\ 0 & \eta
\end{pmatrix}
\end{aligned}
$$
where $\omega_p$ denotes the $p$-adic cyclotomic character. Here, the symbol $a_l(h)$ denotes the $l$-coefficient of the $q$-expasion attached to $h$. Now, since $l\neq p$, we deduce that the trace is well defined at the Frobenius at $l$ and then:
$$\mbox{Trace}(\rho_h)(\mbox{Frob}_l)=(\omega_p (\mbox{Frob}_l)+1)\eta(\mbox{Frob}_l)=(l+1)h(U_l).$$
Since the representation $\rho_{l\text{-new}}^{\text{univ}}$ is defined as the product $\prod_{i\in \tilde{J}} \rho_{h_i}$, it follows that the semi-simplification of its restriction to the decomposition group $G_{\mathbb{Q}_l}$ is unramified, and so the claim follows. 
\end{proof}

\begin{remark}\normalfont
The key role in the above proof is played by the local property (at $l$) of the Galois representation attached to a newform (at $l$) which consists of being ordinary at $l$. We refer the reader to section 5 in \cite{Rib94} for a very clear explanation on the reason why the property holds.
\end{remark}
Now, from the above lemma we deduce that if $g: \mathbb{T}_{N;l} \rightarrow A$ is a cusp eigenform of level $lN$, new at $l$ with weight $k$, character $\chi$ and coefficients in $A$, then, as already mentioned above, its Galois representation $\rho_g$ factors through the universal representation $\rho^{\text{univ}}_{l\text{-new}}$. Hence, denoting by $\tilde{g}$ the map $\mathbb{T}^{l\text{-new}}_{\mathcal{N}^{l\text{-new}}}\rightarrow A$ induced by $g$, and denoting by $\tilde{g}_{\star}$ the map induced on the general linear groups by $\tilde{g}$, we can apply the trace map and get:
$$\mathrm{Trace}(\rho_g)(\mathrm{Frob}_l)=\mathrm{Trace}(\tilde{g}_\star(\rho^{\text{univ}}_{l\text{-new}}))(\mathrm{Frob}_l)=\tilde{g}((l+1)U_l)=(l+1)g(U_l).$$
Note that even though $\rho_g$ is a priori not unramified at $l$, the image of the restriction at the $l$-decomposition group is contained in a Borel subgroup and the two characters on the main diagonal are unramified at $l$, hence its semisimplification is unramified as well and so the trace and the determinant are well defined at Frob$_l$.\\
A key fact in proving the necessity of the level raising condition lies in the explicit description of the $l$-th coefficient of a classical newform of level $lN$ , and now we will prove how this knowledge can be extended in the more general case of cuspidal eigenforms modulo prime powers. More specifically, Hida proved (see Lemma 3.2, \cite{Hida1985}) that if $g$ is a classical newform (at $l$) of level $lN$ (as usual, $l\nmid N$), weight $k\geq 2$, and character $\chi$ of conductor which divides $N$, then $a_l(g)^2=l^{k-2}\chi(l)$, where as usual $a_l(g)$ denotes the $l$-coefficient of the $q$-expansion of $g$. In other words, every  $\mathcal{O}_\mathbb{K}$-algebra homomorphism $g:\mathbb{T}_{N;l}\rightarrow\mathcal{O}_\mathbb{K}$ which is new at $l$ satisfies $g(U_l^2-S_l)=0$. We observe that we can apply Hida's lemma because we are working with the congruence subgroup $\Gamma_1(N)\cap\Gamma_0(l)$ and so the conductor of $\chi$ is coprime to $l$.\\
 This result can be easily generalized to cuspidal eigenforms with coefficients in $A$ of level $lN$ which are new at $l$ by observing that, thanks to the result of Hida mentioned before, the operator $U_l^2-S_l$ is in the kernel of the natural projection $\mathbb{T}_{N;l}\twoheadrightarrow \mathbb{T}_{N;l}^{l\text{-new}} $.\\
 Hence, we have the following lemma:
 
\begin{lemma}
\label{LemmaHida}
 	Let $g: \mathbb{T}_{N;l}\rightarrow A$ a cuspidal eigenform of weight $k\geq2$, level $lN$, character $\chi$ and coefficients in $A$, which is new at $l$. Then $g$ satisfies $g(U_l^2-S_l)=0$.
\end{lemma}
The proof is immediate from the definition of $g$ being new at $l$ since the operator $U_l ^2 -S_l$ belongs to the kernel of the natural projection $\mathbb{T}_{N,l}\twoheadrightarrow \mathbb{T}_{N,l}^{l\text{-new}}$. It is also worth to mention that this result can be expressed in terms of properties of the Galois representation attached to a cuspidal eigenform new at $l$ with coefficients in the $\mathcal{O}_\mathbb{K}$-algebra $A$. Indeed, it is enough to observe that by construction:

$$\mathrm{Det}(\rho_{l\text{-new}}^{\mathrm{univ}})(\mathrm{Frob}_l)=l U_l^2=l S_l$$
where the last equality holds because of Hida's result. We recall that the determinant is well defined at Frob$_l$ since the semi-simplification of the restriction of $\rho_{l\text{-new}}^{\text{univ}}$ to the decomposition group $G_{\mathbb{Q}_l}$ is unramified at $l$. Hence, Lemma \ref{LemmaTrace} and Lemma \ref{LemmaHida} give the explicit formula for the characteristic polynomial at $\text{Frob}_l$: 
$$\text{charpol}((\rho_{l\text{-new}}^{\text{univ}})(\text{Frob}_l))(x)=x^2-(l+1)U_l x+l S_l.$$
We recall that, for the sake of simplicity, we are making an abuse of notation in using the same symbol $U_l$ for the $l$-th Hecke operator in the Hecke algebra $\mathbb{T}$ of level $lN$, and for the images of $U_l$ in the completion $\mathbb{T}_\mathcal{N}$ and $\mathbb{T}^{l\text{-new}}_{\mathcal{N}^{l\text{-new}}}$.\\
Finally, we ready to use the theory of Galois representations to compare the coefficients of cuspidal eigenforms modulo prime powers in a context of level raising.
\subsection{The level raising condition}
In this section we will find a necessary condition for two cusp eigenforms modulo $\pi^r$ (one of level $N$ and one of level $lN$, new at $l$ where $l$ does not divide $pN$) to be equal (modulo $\pi^r$) almost everywhere. We will do it by comparing the associated Galois representations constructed in the previous subsection.\\
Let $A:=\mathcal{O}_\mathbb{K}/(\pi^r)$ for some positive integer $r$.
Let $f$ be a cusp eigenform of level $N$, weight $k$, character $\chi$ and coefficients in $A$. As usual let $l$ be a prime not dividing $pN$. We will always assume that the semisimple residual Galois representation $\rho_{\bar{f}}$ attached to $f$ is absolutely irreducible. Then, by Carayol's theorem (see Thm. 3, \cite{Carayol1994}) mentioned above, we have that $\rho_{f}$ is unramified at $l$ and it satisfies:
$$\mbox{Trace}(\rho_f(\mbox{Frob}_l))=f(T_l).$$
As in the previous section, let $g$ be a cusp eigenform of level $lN$, new at $l$, same weight and character of $f$ with coefficients in $A$. By the discussion in the previous section, we know that the Galois representation $\rho_g$ attached to $g$ through Carayol's theorem satisfies the well-defined property:
$$\mathrm{Trace}(\rho_g)(\mathrm{Frob}_l)=(l+1)g(U_l).$$
Hence, if we assume that $f$ is congruent to $g$, i.e. $f(T_n)=g(T_n)$ for all positive integers $n$ coprime with $lpN$, we have that in terms of Galois representations:
$$\mbox{Trace}(\rho_f(\mbox{Frob}_q))=\mbox{Trace}(\rho_g(\mbox{Frob}_q))$$
for all primes $q$ not dividing $plN$. \\
Hence, by the Chebotarev's density theorem, the representations $\rho_{f}$ and $\rho_g$ agree on a subset of density one of the set of generators of the absolute Galois group, hence they agree at all the traces, in particular at $l$. Thanks to a theorem of Carayol (see Thm.1, \cite{Carayol1994}), the two representations $\rho_f$ and $\rho_g$ are isomorphic. In particular,
$$\mbox{Trace}(\rho_f(\mbox{Frob}_l))=\mbox{Trace}(\rho_g(\mbox{Frob}_l)),$$
where $\mbox{Trace}(\rho_g(\mbox{Frob}_l))$ is well defined since the characters on the diagonal of the restriction of $\rho_g$ at the decompisition group at $l$ are unramified.
More explicitly, we have
$$f(T_l)=(l+1)g(U_l) \text{   in }\mathcal{O}_\mathbb{K}/(\pi^r),$$
where $T_l\in\mathbb{T}_N$ and $U_l\in\mathbb{T}_{N;l}$. Moreover, we can make the above condition independent of $g$ and get the necessary so-called level raising condition. In order to do so, we first need to introduce a special operator in the Hecke algebra $\mathbb{T}_{N;l}$. Assume that the field $\mathbb{K}$ is sufficiently big in the sense that it contains the roots of the polynomial $x^2-\chi(l)$ and, if the weight $k$ is odd, it contains a square root of $l$, then the following lemma holds:

\begin{lemma}
There exists an operator $R_l$ in the $\mathcal{O}_\mathbb{K}$-algebra $\mathbb{T}_{N;l}$ such that $R_l^2=S_l$. 
\end{lemma}

\begin{proof}
Indeed, fix an odd positive integer $m$ and fix once and for all a root $\zeta \in \mathbb{K}$ of the polynomial $x^2-\chi(l)$.\\
We define the operator $R_l:=(l^{\frac{2-k}{2}} \zeta^{-1})^{m} S_l^{\frac{m+1}{2}}$ (see also sec. 4, \cite{Diamond1989} for a similar definition in weight $k=2$). A straightforward computation shows that $R_l^2=S_l$. Note that if $g:\mathbb{T}_{N;l}\rightarrow A$ is a cusp eigenform then $g(R_l)=l^{\frac{k-2}{2}} \zeta$, and so the operator $R_l$ does not depend on the choice of $m$. We need to prove that $R_l$ lies in the $\mathcal{O}_\mathbb{K}$-algebra $\mathbb{T}_{N;l}$.\\
Since $m$ is odd and $S_l \in \mathbb{T}_{N;l}$ we deduce that $S_l^{\frac{m+1}{2}}\in\mathbb{T}_{N;l}$. Since $\chi(l)$ is a root of unity and $\zeta\in\mathbb{K}$ we deduce that $\zeta\in\mathcal{O}_{\mathbb{K}}^{\times}$, so same holds for $\zeta^{-m}$. Since $l^{\frac{k-2}{2}}\in \mathcal{O}_{\mathbb{K}}^{\times}$ we conclude that $R_l\in\mathbb{T}_{N;l}$. \\
Indeed, since $g(U_l)^2=g(S_l)$, the assumption $p\neq 2$ and Hensel's lemma ensure that there exists an $\epsilon_{l,g}\in\{\pm 1\}$ such that $g(U_l)=\epsilon_{l,g} g(R_l)=\epsilon_{l,g} l^{\frac{k-2}{2}} \zeta$. We recall that $\zeta\in\mathcal{O}^{\times}_\mathbb{K}$ satisfies $\zeta^2=\chi(l)$ and it exists because, by assumption, $\mathbb{K}$ contains the splitting field of $p_{l,\chi}(x)=x^2-\chi(l)=0$. 
\end{proof}
Since $g$ has the same weight $k$, and same character $\chi$ of $f$, we finally conclude that:
$$f(T_l)=\epsilon_{l,g}(l+1)f(R_l)=\epsilon_{l,g}(l+1)l^{\frac{k-2}{2}}\zeta \;\;\;\;\;\;\;\text{ in }\mathcal{O}_\mathbb{K}/(\pi^r).$$
Denoting the kernel of $f: \mathbb{T}_N \rightarrow \mathcal{O}_{\mathbb{K}}/{(\pi^r)}$ by $I$, we then restate the above necessary property for the level raising as: for $\epsilon \in\{\pm 1\}$
$$ T_l-\epsilon(l+1)R_l \in I\;\;\;\;\;\;\;\;\;\;\;\;\;\;\;\; (\mathrm{LRC})_\epsilon$$
We will refer to the above condition by ``\textit{level raising condition of parameter} $\epsilon$", or shortly (LRC)$_\epsilon$.\\
Then we have proven the following necessary condition for the level raising in the context of newforms modulo prime powers:
 
 \begin{proposition} \label{necessary}
 	Let $f:\mathbb{T}_N\rightarrow \mathcal{O}_{\mathbb{K}}/(\pi^r)$ be a cusp eigenform of weight $k\geq 2$, level $N\in\mathbb{Z}_{>0}$, character $\chi$ where $\mathbb{K}$ is a sufficiently big, finite extension of $\mathbb{Q}_p$. Assume that $\rho_{\bar{f}}$ is absolutely irreducible. Let $l$ be a prime such that $l\nmid pN$. Suppose there exists a cusp eigenform $g:\mathbb{T}_{N;l}\rightarrow \mathcal{O}_{\mathbb{K}}/(\pi^r)$ of level $lN$, same weight and character as $f$, which is new at $l$ and it satisfies: 
 	$$f(T_n)=g(T_n) \mbox{ for all } n\in\mathbb{N} \mbox{ such that } \mbox{gcd}(n, plN)=1$$
 	then $f$ satisfies $\mathrm{(LRC)}_\epsilon$ for some $\epsilon\in\{\pm1\}$.
 \end{proposition}
 
\begin{remark}
\normalfont
\label{suffbig}
With $\mathbb{K}$ sufficiently big we mean a finite extension of $\mathbb{Q}_p$ which contains a square root of $\chi(l)$ and, if $k$ is odd, it contains a square root of $l$. 
\end{remark}
 
 We want to remark that in the literature (e.g. \cite{Ribet1990}, \cite{Diamond1991}, \cite{Wes05}) the level raising condition is usually expressed as follows:
 $$T^2_l-(l+1)^2 S_l\in I \;\;\;\;\;\;\;\;\;\;\;\; \mathrm{(LRC)}^2.$$
 We will refer to the above condition as $\mathrm{(LRC)}^2$. Clearly, if $f$ is a cusp eigenform modulo $\pi^r$ which satisfies $\mathrm{(LRC)}_\epsilon$, then it satisfies also $\mathrm{(LRC)}^2$. The converse does not hold in general. Before stating the result connecting the two level raising conditions we need some definitions. First, we will denote by $v_\pi$ the $\pi$-adic valuation on $\mathcal{O}_\mathbb{K}$. Now, let $f:\mathbb{T}_N \rightarrow \mathcal{O}_\mathbb{K}/(\pi^r)$ be a cusp eigenform of weight $k$ and character $\chi$. Let $l$ be a prime which does not divide $pN$ such that $f$ satisfies the level raising condition at $l$ given by $\mathrm{(LRC)}^2$. Let $s\leq r-1$ be a positive integer. As usual we denote by $I$ the kernel of $f$. We define $I_s$ as the kernel of the composition $f_s :\mathbb{T}_N\rightarrow \mathcal{O}_\mathbb{K}/(\pi^r)\twoheadrightarrow \mathcal{O}_\mathbb{K}/(\pi^s)$ where the first arrow is given by $f$ and the second one is given by the natural projection. We say that $f$ satisfies $s$-$\mathrm{(LRC)}_\epsilon$ if there exists an $\epsilon\in\{\pm1\}$ such that the reduction of $f$ mod $\pi^s$, i.e. $f_s$, satisfies 
 $f_s(T_l)-\epsilon(l+1) f_s(R_l) \in I_s$. For instance, the conditions $r$-$\mathrm{(LRC)}_\epsilon$ and $\mathrm{(LRC)}_\epsilon$ are the same. The condition $0$-$\mathrm{(LRC)}_\epsilon$ is the empty condition. Now we are ready to present the following result which states a more precise connection between the level raising conditions. We recall that we are assuming that $p$ is an odd prime. 

\begin{lemma}
	\label{confronto}
	Let $f:\mathbb{T}_N \rightarrow \mathcal{O}_\mathbb{K}/(\pi^r)$ a cusp eigenform of weight $k$ and character $\chi$. Let $l$ be a prime which does not divide $pN$ such that $f$ satisfies the level raising condition at $l$ given by $\mathrm{(LRC)}^2$. Let $v$ be the value of the $\pi$-adic valuation at $l+1$, i.e. $v=v_\pi(l+1)$.\\
	 Then there exists an $\epsilon\in\{\pm1\}$ such that $f$ satisfies the following level raising conditions $(r-v)\text{-}\mathrm{(LRC)}_\epsilon$.
	 In particular, if the prime $l$ satisfies $v_\pi(l+1)= 0$, then the conditions $\mathrm{(LRC)}^2$ and $\mathrm{(LRC)}_\epsilon$ for some $\epsilon \in\{\pm 1\}$ are equivalent.

\end{lemma}

\begin{proof}
	The result comes from straightforward computations modulo prime powers. Consider the identity $f(T_l )^2-(l+1)^2 f(S_l )=(f(T_l )-(l+1)f(R_l ))(f(T_l )+(l+1)f(R_l ))$ and define 
	$a:=v_\pi(f(T_l )-(l+1)f(R_l ))$ and $b:=v_\pi (f(T_l )+(l+1)f(R_l ))$. We can assume, without loss of generality, that $a\leq b$. Since $v_\pi(2 f(R_l ))=0$, we have that 
	$$v=v_\pi(l+1)=v_\pi (f(T_l )-(l+1)f(R_l )-f(T_l )-(l+1)f(R_l ))\geq \text{min}\{a,b\}=a, $$
	and so we deduce that $b=r-a \geq r-v$. In other words, the cuspidal eigenform $f$ satisfies $(r-v)$-(LRC$)_{-}$.

\end{proof}
\subsection{The level raising congruence}
As in the previous section, let $N\geq 5$, $r$ and $k\geq 2$ be positive integers. Fix an odd prime $p$ not dividing $N$. Let $\mathbb{K}$ be a sufficiently big finite extension of $\mathbb{Q}_p$, with ring of integers $\mathcal{O}_\mathbb{K}$ and a uniformizer $\pi$.\\
We will prove the following:

\begin{theorem}
	\label{MainTheorem}
	Let $f:\mathbb{T}_N\rightarrow \mathcal{O}_\mathbb{K}/(\pi^r)$ be a cusp eigenform of level $N$, weight $k$ and character $\chi$. Assume that its associated residual Galois representation $\bar{\rho}_{f}$ is absolutely irreducible. Suppose that $p$ does not divide $\varphi(N)N(k-2)!$.\\
If $k=p$ or $k=p+1$, assume that the localized Hecke algebra $\mathbb{T}_\mathcal{M}$ is Gorenstein, where $\mathcal{M}$ is the kernel of the reduction of $f$ modulo $\pi$.\\
	Let $l$ be a prime which does not divide $pN$. Then the two following statements are equivalent:
	
	$$
	\begin{aligned}
	(i)&\;f \text{ satisfies the level raising condition } \mathrm{(LRC)}_\epsilon \text{ at the prime } l \text{ for some } \epsilon\in \{\pm1\}\\
	(ii)&\;\text{ there exists a cusp eigenform } g: \mathbb{T}_{N;l} \rightarrow \mathcal{O}_{\mathbb{K}}/{(\pi^r)} \text{ of level } lN, \text{ weight } k \text{ and character } \chi \text{ which satisfies: }
	\end{aligned}
	$$
	$$
	\begin{aligned}
	(ii.1)&\;\; f(T_q)=g(T_q) \;\;\text{for all primes }\; q\nmid lpN,\\
	(ii.2)&\;\; g\text{ is new at }l\text{, i.e. it factors through the quotient }\; \mathbb{T}_{N;l}^{l\text{-new}}. \\
	\end{aligned}
	$$
\end{theorem}

\begin{remark} \normalfont
The theorem specializes to Ribet's and Diamond's results (for $p\not =2,3$ and assuming the Gorenstein condition for $k=p$ and $k=p+1$ for the localized Hecke algebra) when $k\geq 2$ and $r=1$ (see Thm. 1 in \cite{Ribet1990}, and Thm. 1 in \cite{Diamond1991}) thanks to the Deligne-Serre's lifting lemma (see Lemma 6.11, \cite{DelSer74}) and Carayol's lemma (see Prop. 1.10 \cite{Edi95}). As a consequence of Carayol's lemma, we can still recover Diamond's theorem in the case $p=3$ and non-trivial character if some extra conditions on $\chi$ hold.\\
Moreover, under the restriction $p\not =2$, it specializes to Tsaknias-Wiese's result when $k=2$ and $r \geq 1$ (see Thm. 5 in \cite{TsakniasWiese2017}).
\end{remark}

\begin{remark} \normalfont
The Gorenstein property for the localized Hecke algebra at some maximal ideal of characteristic $p$ (odd prime) arises naturally in the context of residual modular Galois representations to study multiplicity one questions and it was extensively studied by several authors. \\
Being Gorenstein is a property of Noetherian local rings and as such it can be studied from a purely algebraic point of view (see \cite{Til95}). Concerning Hecke algebras, for weight $k=2$, it is always true that the localization at a maximal ideal (whose characteristic $p$ does not divide the level $N$) is Gorenstein (see Thm. 9.2, \cite{Edi92}). \\
The case of general weight $2\leq k\leq p-1$, was settled by a theorem of Faltings and Jordan (see \cite{FJ95}).\\
The cases corresponding to weights $k=p$ or $k=p+1$ are more complicated. One can still find sufficient conditions for the Gorenstein property to hold in the article of Edixhoven (see \cite{Edi92}); but there are known counterexamples due to Kilford and Wiese (see \cite{KW08}) and theoretical results, due to Wiese (see \cite{Wiese07}), that establish the existence of counterexamples in general. Moreover, in \cite{KW08}, a notion of Gorenstein defect is defined and studied.
\end{remark} 


\begin{remark} \normalfont
Here we summarize the key points of the proof of Theorem \ref{MainTheorem}, which will be given in details in the next section.\\
In order to raise the level of a cusp eigenform $f:\mathbb{T}_N\rightarrow \mathcal{O}_\mathbb{K}/ \pi^r$ we first construct a $\mathbb{T}_N$-module $W_f$ such that $\text{Ann}_\mathbb{T}(W_f)=\text{Ker}(f)$ and the action is given by $T \cdot w=f(T)w$ for all $w\in W_f$ and for all $T\in\mathbb{T}_N$. The $\mathbb{T}_N$-module structure of $W_f$ determines uniquely the cuspidal eigenform $f$ (this will be a consequence of lemma \ref{bah}).\\
By a (slightly adapted) result of Diamond (see Lemma 3.2, \cite{Diamond1991}) and assuming the level raising condition $(LRC)_\epsilon$ for some $\epsilon\in \{\pm 1\}$, it is possible to associate to $W_f$ a $\mathbb{T}_{N;l}$-module $V_f$ such that $\mathbb{T}_{N;l}/{\text{Ann}_{\mathbb{T}_{N;l}}(V_f)}$ is isomorphic (as $\mathcal{O}_\mathbb{K}$-algebra) to $\mathcal{O}_\mathbb{K}/(\pi^r)$. The structure homomorphism of $V_f$ as $\mathbb{T}_{N;l}$-module is given by a $\mathcal{O}_\mathbb{K}$-algebra homomorphism $g: \mathbb{T}_{N;l}\rightarrow \text{End}_{\mathcal{O}_\mathbb{K}}(V_f)$ which factors via the natural projection (we keep calling it $g$) $g: \mathbb{T}_{N;l}\rightarrow \mathbb{T}_{N;l}/{\text{Ann}_{\mathbb{T}_{N;l}}(V_f)}$.  By construction, the $\mathcal{O}_\mathbb{K}$-algebra homomorphism $g:\mathbb{T}_{N;l}\rightarrow {\mathcal{O}_\mathbb{K}}/(\pi^r)$ (i.e. a cuspidal eigenform modulo $\pi^r$ of level $lN$) will satisfy $f(T_n)=g(T_n)$ for all positive integers $n$ coprime with $lpN$.\\
Finally, we will prove that $g$ is new at $l$, i.e. it factors through $\mathbb{T}_{N;l}^{l\text{-new}}$. Indeed, assuming the level raising condition $(LRC)_\epsilon$, we prove that $U_{l}-\epsilon R_l$ is contained in $\text{Ann}_{\mathbb{T}_{N;l}} (V_f)$ and, as a consequence, $V_f$ is a $\mathbb{T}_{N,l}$-submodule of a non-trivial congruence module $\Omega$, i.e. a $\mathbb{T}_{N;l}$-module whose annihilator ${\text{Ann}_{\mathbb{T}_{N;l}}}(\Omega)$ contains both the kernels of the natural projections of the Hecke algebra $\mathbb{T}_{N;l}$ onto its $l$-old and $l$-new parts. The existence of such congruence module is granted by the work of Diamond (see \cite{Diamond1991}). Hence, we have that $\text{Ann}_{\mathbb{T}_{N;l}}(\Omega)\subseteq \text{Ann}_{\mathbb{T}_{N;l}} (V_f)$. It follows that $\text{Ker}(\mathbb{T}_{N,l} \twoheadrightarrow \mathbb{T}_{N,l}^{l\text{-new}} )$ is contained in $\text{Ann}_{\mathbb{T}_{N;l}} (V_f)$, and so in particular $g$ factors through the quotient $\mathbb{T}_{N;l}^{l\text{-new}}$, i.e. it is new at $l$.
\end{remark}

\subsection{Proof of Theorem \ref{MainTheorem}}
The necessity of the level raising condition is exactly the content of Proposition \ref{necessary}. 
Hence, we need to prove that if $f$ satisfies $(LRC)_\epsilon$ for some $\epsilon\in\{\pm\}$ then there exists a cusp eigenform $g: \mathbb{T}_{N;l} \rightarrow \mathcal{O}_{\mathbb{K}}/{(\pi^r)}$ such that $f(T_n)=g(T_n)$ for every
$n$ such that $(n,lN)=1$ and such that $g$ factors through $\mathbb{T}^{l\text{-new}}_{N;l}$.\\ 
Denote the kernel of $f: \mathbb{T}_N \rightarrow \mathcal{O}_{\mathbb{K}}/{(\pi^r)}$ as $I$. 
Assume that 
$$T_l-\epsilon(l+1)R_l\in I \;\;\;\;\;\;\;\;\;\;\;\; (\mathrm{LRC})_\epsilon$$ 
for some $\epsilon\in\{\pm1\}$. 
Denote by $\bar{f}$ the reduction of $f$ modulo $\pi$ and denote by $\mathcal{M}$ its kernel. We have that $\mathcal{M}$ is a maximal ideal and $I\subseteq \mathcal{M}$.\\
For $k\geq 2$, and for $\Gamma$ any congruence subgroups of  $\text{Sl}_2(\mathbb{Z})$, let $L_{k-2}(\mathcal{O}_\mathbb{K})$ denote the $\Gamma$-module $\mathrm{Sym}^{k-2} (\mathcal{O}_\mathbb{K} ^2)$. The $\Gamma$-module structure of $L_{k-2}(\mathcal{O}_\mathbb{K})$ is induced by the symmetric power of the module structure of $\mathcal{O}_\mathbb{K} ^2$ given by standard matrix multiplication.\\
The parabolic cohomology group $H^1_P(\Gamma_1(N),L_{k-2}(\mathcal{O}_{\mathbb{K}}))$ is obtained as a subgroup of the standard cohomology group $H^1(\Gamma_1(N),L_{k-2}(\mathcal{O}_{\mathbb{K}}))$ by considering the cocycles $w$ satisfying $w(\gamma)\in (\gamma-1)L_{k-2}(\mathcal{O}_\mathbb{K})$ for all parabolic elements $\gamma\in\Gamma_1(N)$, i.e. all the matrices conjugate to $\pm \begin{pmatrix} 1 & d\\ 0 & 1\end{pmatrix}$ for some integer $d$ (see sec. 1.4 \cite{Shi71}).
Recalling that $l$ does not divide $pN$, we define:

$$W(\mathcal{O}_{\mathbb{K}})=\text{Image}( H^1_P(\Gamma_1(N),L_{k-2}(\mathcal{O}_{\mathbb{K}}))\xrightarrow{j_1} H^1_P(\Gamma_1(N),L_{k-2}(\mathbb{K}))) $$
$$V(\mathcal{O}_{\mathbb{K}})=\text{Image}(H^1_P(\Gamma_1(N)\cap \Gamma_0(l),L_{k-2}(\mathcal{O}_{\mathbb{K}}))) \xrightarrow{j_2} H^1_P(\Gamma_1(N)\cap \Gamma_0(l),L_{k-2}(\mathbb{K})))$$
where the maps $j_1$ and $j_2$ are the ones induced on cohomology by the natural injection of $\mathcal{O}_{\mathbb{K}}$ in $\mathbb{K}$.\\
Consider the $\mathcal{O}_{\mathbb{K}}$-module $\mathbb{K}/\mathcal{O}_\mathbb{K}$, we define:
$ W(\mathbb{K}/\mathcal{O}_\mathbb{K}):=W(\mathcal{O}_{\mathbb{K}})\otimes_{\mathcal{O}_{\mathbb{K}}} \mathbb{K}/\mathcal{O}_\mathbb{K}$.\\ 
By Hida (see Thm. 3.2, \cite{Hi81}, we recall that we are under the assumption $N\geq 5$), assuming that $p\nmid N(k-2)!$ (or equivalently, $p\nmid N$ and $k\leq p+1$) the $\mathcal{O}_\mathbb{K}$-module $W(\mathcal{O}_\mathbb{K})$ (resp. $V(\mathcal{O}_\mathbb{K})$) is finite, free and self-dual with respect to the perfect pairing given by the cup product.\\
The double coset operators of $\text{Sl}_2(\mathbb{Z})$ act on the $\mathcal{O}_\mathbb{K}$-module $W(\mathcal{O}_\mathbb{K})$ and such an action is compatible with the action of the double coset operators on the space of classical cuspidal eigenforms of weight $k$ and level $N$ (see sec. 8.3 in \cite{Shi71}).   In other words, Eichler and Shimura (see sec. 8.3, \cite{Shi71}) proved that the $\mathcal{O}_\mathbb{K}$-module $W(\mathcal{O}_\mathbb{K})$ (resp. $V(\mathcal{O}_\mathbb{K})$) is invariant under the action of the full Hecke algebra acting on the space of cusp forms of weight $k\geq2$ and level $\Gamma_1(N)$ (resp. level $\Gamma_1(N)\cap\Gamma_0(l)$). Now, let $\alpha$ be the map
 $$\alpha: W(\mathbb{K}/\mathcal{O}_{\mathbb{K}})^{\oplus2}\rightarrow V(\mathbb{K}/\mathcal{O}_{\mathbb{K}})$$
 induced by the inclusions of $\Gamma_1(N)\cap \Gamma_0(l)$ in $\Gamma_1(N)$ via the identity and the conjugation by $\begin{pmatrix}
 l & 0 \\ 0 & 1
 \end{pmatrix}$.
  As a generalization of Ihara's lemma (see lemma 3.2, \cite{Ihara1975}), Diamond proved (see lemma 3.2, \cite{Diamond1991}) the following:
  
\begin{lemma}(Diamond) \label{Diamond} The map $\alpha$ is injective.
\end{lemma}

Diamond proved this lemma, which is a fundamental step in the proof of level raising for classical cusp eigenforms of weight $k\geq2$, to deduce that if the level raising condition holds then there exists a non-trivial congruence module $\Omega$ in $V(\mathbb{K}/\mathcal{O}_{\mathbb{K}})$, i.e. $\Omega$ is a non-trivial $\mathbb{T}_{N;l}$-submodule of $V(\mathbb{K}/\mathcal{O}_{\mathbb{K}})$ whose action of $\mathbb{T}_{N;l}$ factors through $\mathbb{T}_{N;l}^{l\text{-new}}$ and $\mathbb{T}_{N;l}^{l\text{-old}}$. Moreover, there is an isomorphism $\Omega\cong \text{Ker}(\beta\circ\alpha)$ where $\beta$ is the adjoint map of $\alpha$ given by the standard cup product on the cohomology groups.\\
Since we want to prove that the level raising congruence can be obtained preserving the Dirichlet character $\chi$ associated to the diamond operators, we will first restrict ourself to the $\chi$-invariant submodules of the parabolic cohomology groups defined above. \\ 
Indeed, the $\mathcal{O}_\mathbb{K}$-module $W(\mathcal{O}_\mathbb{K})$ has a natural action of the group associated to diamond operators $\Gamma_0(N)/\Gamma_1(N)\cong (\mathbb{Z}/N\mathbb{Z})^{\times}$ (via double coset operators, see chap. 8 in \cite{Shi71}). For any Dirichlet character $\psi: (\mathbb{Z}/N\mathbb{Z})^\times\rightarrow (\mathcal{O}_\mathbb{K})^{\times}$, we define
$$W(\mathcal{O}_\mathbb{K})^{(\psi)}:=\{v\in W(\mathcal{O}_\mathbb{K}): h \cdot v= \psi(h) v \;\;\; \text{ for all } \;h\in(\mathbb{Z}/N\mathbb{Z})^{\times}\}.$$
Morever, we define $W(\mathbb{K}/\mathcal{O}_\mathbb{K})^{(\psi)}:=W(\mathcal{O}_{\mathbb{K}})^{(\psi)}\otimes_{\mathcal{O}_{\mathbb{K}}} \mathbb{K}/\mathcal{O}_\mathbb{K}$. The  $\mathcal{O}_\mathbb{K}$-submodule $W(\mathbb{K}/\mathcal{O}_\mathbb{K})^{(\psi)}$ (resp. $V(\mathbb{K}/\mathcal{O}_\mathbb{K})^{(\psi)}$) inherits a natural structure of $\mathbb{T}_{N}$-module (resp. $\mathbb{T}_{N;l}$-module).\\
The assumption that the odd prime $p$ does not divide $\varphi(N)$ ensure us that the Hecke submodule $W(\mathcal{O}_\mathbb{K})^{(\psi)}$ of cocycle classes on which the group $(\mathbb{Z}/N\mathbb{Z})^{\times}$ acts via the Dirichlet character $\psi$ is a direct summand of $W(\mathcal{O}_\mathbb{K})$. In order to see this, it is enough to observe that that there exist, in the $\mathcal{O}_\mathbb{K}$-algebra $\mathbb{T}_N$, the idempotent $e_\psi : W(\mathcal{O}_\mathbb{K})\rightarrow W(\mathcal{O}_\mathbb{K})^{(\psi)}$ given by $e_\psi=\frac{1}{\varphi(N)}\sum_{g\in (\mathbb{Z}/N\mathbb{Z})^{\times}} \psi(g^{-1})g$.\\
The theory of cohomological congruence module developed by Diamond (see \cite{Diamond1991}) works with fixed character once assumed that $p$ does not divide $\varphi(N)$.\\
Indeed, under the extra hypothesis $p\nmid \varphi(N)$, the restriction of $\alpha$ to $(W({\mathbb{K}/\mathcal{O}_{\mathbb{K}}})^{(\chi)})^{\oplus2}$ gives us an injective map
$$\alpha^{(\chi)}: (W({\mathbb{K}/\mathcal{O}_{\mathbb{K}}})^{(\chi)})^{\oplus2}\rightarrow V(\mathbb{K}/\mathcal{O}_{\mathbb{K}})^{(\chi)}.$$
As for $\alpha$, we can associate to the map $\alpha^{(\chi)}$ a well defined congruence module $\Omega^{(\chi)}\cong\text{Ker}(\beta^{(\chi)}\circ \alpha^{(\chi)})$, where $\beta^{(\chi)}$ is the adjoint map associated to $\alpha^{(\chi)}$ with respect to the cup product (see also sec. 3, \cite{Diamond1989}).\\
Now, we define  $W_f:=W(\mathbb{K}/\mathcal{O}_{\mathbb{K}})[I]$. We recall that if $R$ is a commutative ring, $I$ an ideal and $M$ is an $R$-module, we define $M[I]=\{m\in M: \; Im=0\}$. \\
The $\mathcal{O}_\mathbb{K}$-module $W_f$ has a natural structure of $\mathbb{T}_N$-module, and it is contained in $W(\mathbb{K}/\mathcal{O}_\mathbb{K})^{(\chi)}$.\\
Hence, the assumption $p\nmid \varphi(N)$ allows us to restrict ourself to work with Hecke submodules of $W(\mathbb{K}/\mathcal{O}_\mathbb{K})$ on which the diamond operators act via the character $\chi$. \\
By these considerations, from now on, we will simplify the notation by dropping the superscript for the fixed character $\chi$; in other words, from now on $\alpha=\alpha^{(\chi)}$ and similar for $W(\mathbb{K}/\mathcal{O}_\mathbb{K})^{(\chi)}$, $V(\mathbb{K}/\mathcal{O}_\mathbb{K})^{(\chi)}$and $\Omega^{(\chi)}$. We can now continue the proof.\\ 
The chain of inclusion $(\pi^r)\subseteq I \subset \mathcal{M}$ induces:

$$ W(\mathbb{K}/\mathcal{O}_{\mathbb{K}})[\pi^r]\supseteq W_f \supseteq W(\mathbb{K}/\mathcal{O}_{\mathbb{K}})[\mathcal{M}].$$
The characteristic of $\mathbb{T}_N/I$ is a power of $p$, in particular we have the following chain of inclusions of ideals in $\mathbb{T}_N$: 
$$(\pi^r) \subseteq I \subseteq \mathcal{M}.$$
The action of $\mathbb{T}_{N}$ on $W_f$ factors through the quotient ${\mathbb{T}_{N}}/{I}$. Because of the Gorenstein assumption, a result in commutative algebra (see Lemma \ref{propspecial}) will ensure us that $W_f$ is a faithful $\mathbb{T}_{N}/I$-module, or in other words 
$\text{Ann}_{\mathbb{T}_{N}}(W_f)=I$. We will prove this in the next section, now we will complete the proof of the main theorem.\\ We can finally apply Ribet's analysis in the cohomological context. We recall that we dropped the notation for the fixed character but we are always working with scalar diamond operators coming from the definition of the Hecke algebras $\mathbb{T}_N$ and $\mathbb{T}_{N;l}$ acting faithfully on the space of cusp forms of character $\chi$ and respective level $N$ and $lN$.\\
Hence, in order to get a $\mathbb{T}_{N;l}$-module which will lead to the level raising of $f$, we will embed $W_f$ in $ W(\mathbb{K}/\mathcal{O}_{\mathbb{K}})^{\oplus2}$ with a slightly modified antidiagonal embedding ``$\mathrm{ad}^{\epsilon}_{k,\chi}$" which will depend on $\mathrm{(LRC)}_\epsilon$ (in the sense that it will depend on $\epsilon$) and then we will apply $\alpha$.\\
Since $R_l$ is invertible, we define the embedding:

$$
\begin{aligned}
\mathrm{ad}^{(\epsilon)}_{k,\chi}: W(\mathbb{K}/\mathcal{O}_{\mathbb{K}}) &\hookrightarrow W(\mathbb{K}/\mathcal{O}_{\mathbb{K}})^{\oplus2}\\
x &\mapsto (x, -\epsilon l^{2-k}R_l x)
\end{aligned}
$$
Note that if $k=2$ and $\chi=1$, the embedding $\text{ad}^{\epsilon}_{2,1}$ (which sends $x$ to $(x,-\epsilon x)$) coincides with the embedding considered by Ribet (see \cite{Ribet1990}) and by Tsaknias and Wiese (see \cite{TsakniasWiese2017}) in the context of analogous modular Jacobian varieties (we recall anyway that Ribet already observed in his article (see \cite{Ribet1990}) that his argument his entirely cohomological). \\
Now, applying $\alpha$ we can define 
$V_f:=\alpha(\mathrm{ad}^{\epsilon}_{k,\chi}(W_f))\subseteq V(\mathbb{K}/\mathcal{O}_{\mathbb{K}})$. First, we want to show that $V_f$ is Hecke stable, i.e. it inherits the structure of $\mathbb{T}_{N;l}$-module from the standard Hecke action of $\mathbb{T}_{N;l}$ on $V(\mathbb{K}/\mathcal{O}_{\mathbb{K}})$. Since $\alpha$ commutes with the action of the Hecke operators $T_n$ with $(n,l)=1$ and with the automorphism $R_l$, we have that $V_f$ is $\mathbb{T}_{N;l}$-stable if it is stable under the action of the operator $U_l$.\\
 The action of the operator $U_l$ on $W(\mathbb{K}/\mathcal{O}_{\mathbb{K}})^{\oplus2}$ coincide with the action of the operator $U_l$ on the $l$-old space of classical cusp forms of level $lN$, i.e. it is given by the $2\times 2$ matrix (see \cite{Diamond1991}):

$$
U_l|_{W(\mathbb{K}/\mathcal{O}_{\mathbb{K}})^{\oplus2}}=\begin{pmatrix}
T_l & l^{k-1} \\ -l^{2-k}S_l & 0\\
\end{pmatrix}.$$
Finally, we can check that $V_f$ is $U_l$-invariant:\\
fix $x\in W_f$ and let $y=\alpha( ad^{(\epsilon)}_{k,\chi}(x))=\alpha((x,-\epsilon l^{2-k}R_l x))$, we have that (since $\alpha$ is injective)
$$U_l(y)=U_l(\alpha(ad^{\epsilon}_{k,\chi}))= \alpha \Bigg(\begin{pmatrix}
T_l & l^{k-1} \\ -l^{2-k}S_l & 0\\
\end{pmatrix}\begin{pmatrix}
x  \\ -\epsilon l^{2-k}R_l x\\
\end{pmatrix}\Bigg)= \alpha \Bigg(\begin{pmatrix}
T_l x -\epsilon lR_l x  \\ -l^{2-k} S_l x\\
\end{pmatrix}\Bigg) $$
Since $\mathrm{(LRC)}_\epsilon$ holds, we have that:
$$ 
\begin{aligned}
\alpha \Bigg(\begin{pmatrix}
T_l x -\epsilon lR_l x  \\ -l^{2-k} S_l x\\
\end{pmatrix}\Bigg)&=  \alpha \Bigg(\begin{pmatrix}
\epsilon (l+1)R_l x -\epsilon lR_l x  \\ -l^{2-k} {R_l}^2 \epsilon^2 x\\
\end{pmatrix}\Bigg) \\&=  \alpha \Bigg(\begin{pmatrix}
\epsilon R_l & 0  \\ 0 & \epsilon R_l\\
\end{pmatrix}\begin{pmatrix}
 x  \\ -\epsilon l^{2-k} R_l x\\
\end{pmatrix}\Bigg)\\&= \epsilon R_l y.  
\end{aligned}
$$
The operator $R_l$ is a scalar, and so $V_f$ is a well-defined $\mathbb{T}_{N;l}$-module. \\
Hence, the action of $\mathbb{T}_{N;l}$ on $V_f$ is represented by a $\mathcal{O}_\mathbb{K}$-algebra homomorphism 
$g: \mathbb{T}_{N;l} \rightarrow \mathcal{O}_\mathbb{K}/(\pi^r).$ The last thing to prove is that $g$ factors through $\mathbb{T}^{l\text{-new}}_{N;l}$. Since the raising level condition is satisfied, there exists a non-trivial congruence module $\Omega\subset V(\mathbb{K}/\mathcal{O}_{\mathbb{K}})$, i.e. a $\mathbb{T}_{N;l}$-module whose $\mathbb{T}_{N;l}$-action factors both through $\mathbb{T}^{l\text{-old}}_{N;l}$ and $\mathbb{T}^{l\text{-new}}_{N;l}$. By construction,  we have that the action of $\mathbb{T}_{N;l}$ on $V_f$ factors through $\mathbb{T}_{N;l}^{l\text{-old}}$, hence we will prove directly that $V_f\subseteq \Omega$, which concludes the proof.  
Indeed, since $\text{Ker}(\mathbb{T}_{N;l} \rightarrow \mathbb{T}_{N;l}^{l\text{-new}})$ is contained in the annihilator of $\Omega$ (as $\mathbb{T}_{N;l}$ -module), if $V_f \subseteq \Omega$ then the correspondent controvariant inclusion of annihilators proves that the action factors via $\mathbb{T}_{N;l} ^{l\text{-new}}$. \\
The above claim follows directly from the injectivity of $\alpha$. In particular, it follows from lemma \ref{Diamond} that $\Omega \cong \mathrm{Ker}(\beta \circ\alpha)$ where $\beta$ is the adjoint of $\alpha$ with respect to the standard pairing given by the cup product on $W(\mathbb{K}/\mathcal{O}_{\mathbb{K}})$ and on $V(\mathbb{K}/\mathcal{O}_{\mathbb{K}})$. Indeed, explicitly we have (see \cite{Diamond1991}) that the map:

$$\beta \circ \alpha : W(\mathbb{K}/\mathcal{O}_{\mathbb{K}})^{\oplus 2} \rightarrow  W(\mathbb{K}/\mathcal{O}_{\mathbb{K}})^{\oplus 2}$$
is given by the matrix 
$$
\begin{pmatrix}
l+1 & l^{k-2}T_l S_l^{-1} \\
T_l & l^{k-2}(l+1)
\end{pmatrix}
$$
acting on $ W(\mathbb{K}/\mathcal{O}_{\mathbb{K}})^{\oplus 2}$. So in order to show that $V_f$ is contained in $ \Omega=\mbox{Ker}(\beta\circ \alpha)$ we will prove explicitly that if $y\in V_f$ then $(\beta\circ \alpha) (y)=0$. Now, since Lemma \ref{Diamond} holds, we have an isomorphism (by definition) between $\mbox{ad}^{\epsilon}_{k,\chi}(W_f)$ and $V_f$, so any $y\in V_f$ can be written as
$(x,-\epsilon l^{2-k} R_l x)$ for a unique $x\in W_f$. Finally,

$$(\beta \circ \alpha) (y)= \begin{pmatrix}
l+1 & l^{k-2}T_l S_l^{-1} \\
T_l & l^{k-2}(l+1)
\end{pmatrix} 
\begin{pmatrix}
x  \\
-\epsilon l^{2-k} R_l x
\end{pmatrix}=
\begin{pmatrix}
((l+1)-\epsilon {R_l}^{-1}T_l)x \\
(T_l-\epsilon(l+1)R_l)x
\end{pmatrix}=
\begin{pmatrix}
0 \\
0
\end{pmatrix}
$$
where the last step holds because $R_l$ is an automorphism and  the raising level condition holds, i.e.  $T_l-\epsilon(l+1)R_l \in\text{Ann}_\mathbb{T}(W_f)=I=\text{Ker}(f)$. This implies that $g$ factors through $\mathbb{T}_{N;l}^{l-new}$ and the proof is complete.
\begin{remark} \normalfont
Diamond (see Thm. 2 in \cite{Diamond1991}) proved that if $f$ is a classical newform of level $N\geq 5$, weight $k\geq 2$ and character $\chi$, which satisfies the level raising condition $\text{(LRC)}^2$ modulo $\pi^r$ at a prime $l$ (as usual, not dividing $pN$) then there exist a family $\{g_i\}_{i=1}^{M}$, for some positive integer $M$, of $l$-newforms of level $lN$ such that $f\equiv g_i \;\text{mod}\; \pi^{d_i}$ and $\sum_{i=1}^M d_i \geq r$. Using Theorem \ref{MainTheorem} it is possible to deduce more information about the coefficients $d_i$ whose existence was predicted by Diamond. Indeed, for every $i$, the congruence $f\equiv g_i \;\text{mod}\; \pi^{d_i}$ can be seen as a level raising congruence of cuspidal eigenforms modulo $\pi^{d_i}$, and as such we deduce that $\text{(LRC)}_\epsilon$ has to be satisfied modulo $\pi^{d_i}$ for some $\epsilon \in \{ \pm 1\}$. Then we have an upper bound for the coefficients $d_i$:
$$\text{max}_{i}\{d_i\}\leq \text{max}\{v_\pi(a_l(f)-(l+1)l^{\frac{k-2}{2}}\zeta);\;v_\pi(a_l(f)+(l+1)l^{\frac{k-2}{2}}\zeta )\}$$
where, as before, $\zeta\in\mathcal{O}_\mathbb{K}$ satisfies $\zeta^2=\chi(l)$.
\end{remark}
\subsection{A result in commutative algebra}
Let $p$ be an odd prime and let $\mathbb{K}$ be a finite extension of $\mathbb{Q}_p$.
Let $H$ be a finitely generated, free $\mathcal{O}_\mathbb{K}$-module, and let $R \subseteq \text{End}_{\mathcal{O}_\mathbb{K}}(H)$ be a commutative $\mathcal{O}_\mathbb{K}$-subalgebra.\\
 Consider the $\mathcal{O}_\mathbb{K}$-module $W:=H\otimes_{\mathcal{O}_\mathbb{K}} \mathbb{K}/\mathcal{O}_\mathbb{K}$. The ring $R$ naturally acts on $W$ via the left action on $H$, hence $W$ has a natural structure of $R$-module. Assume that $H$ is a (finite) free $R$-module of rank $s$. The following result is probably well known but we were not able to find it in the literature so we will prove it. 

\begin{lemma} \label{propspecial}
	
Let $h$ be an element of the set $\text{Hom}_{\mathcal{O}_\mathbb{K} \text{-Alg}}(R, \mathcal{O}_\mathbb{K} /(\pi^r))$ an let $I=\text{Ker}(h)$ and $\mathcal{M}$ be the unique maximal ideal containing $I$.
Assume that the localization $R_\mathcal{M}$ is a Gorenstein ring, i.e. $\text{Hom}_{\mathcal{O}_\mathbb{K}} (R_\mathcal{M}, \mathcal{O}_\mathbb{K})$ is a free $R_\mathcal{M}$-module of rank 1.\\ Then $W[I]=\{w\in W : Iw=0\}$ is a faithful $R/I$-module, or equivalently $\text{Ann}_{R}(W[I])=I$.

\end{lemma}

\begin{remark} \label{bello} \normalfont 
We can apply the above lemma in the context of cuspidal eigenforms modulo prime powers. Namely $H$ is the parabolic cohomology group that we denoted by $W(\mathcal{O}_\mathbb{K})$, $R$ is the Hecke algebra $\mathbb{T}_N$ and $h=f$, which gives us the faithfulness of $W[I]=W_f$ as $R/I$-module. A study of $W$ in weight 2 (in terms of Jacobians attached to modular curves) can be found in an article of Tilouine (see Thm. 3.4 and cor. (1), \cite{Til95}).
\end{remark}

\begin{proof}
Since $R$ is commutative, the inclusions of ideals $(\pi^r)\subseteq I\subset \mathcal{M}$ induce a chain of inclusions of $R$-modules: 
$$W[\mathcal{M}] \subseteq W[I] \subseteq W[\pi^r].$$ 
Let $\mathrm{Ta}_\pi(W)=\varprojlim W[\pi ^n]$ be the $\pi$-adic Tate module associated to $W$. Since $W[\pi^r]$ and $H/\pi^r H$ are isomorphic as $R$-modules and the action of $R$ on both modules is compatible with the transition maps of the projective sytems, we have a canonical isomorphism of $R$-modules $\text{Ta}_\pi(W)\cong H$. We define $\mathcal{Z}$ as the localization of $\mathrm{Ta}_\pi(W)$ at $\mathcal{M}$.\\
We have that $\mathcal{Z}$ is a finitely generated, free $R_\mathcal{M}$-module of rank $s$.
The localized ring $R_\mathcal{M}$ is Gorenstein, or equivalently there is an isomorphism of $R_{\mathcal{M}}$-modules $R_{\mathcal{M}}\cong \mathrm{Hom}_{\mathcal{O}_\mathbb{K}}(R_\mathcal{M},\mathcal{O}_\mathbb{K})$.
As usual, let $W[\pi^r]_\mathcal{M}$ be the $R_\mathcal{M}$-module obtained by localization of $W[\pi^r]$ at $\mathcal{M}$. We have a chain of isomorphisms of $R_{\mathcal{M}}$-modules:

$$W[\pi^r]_\mathcal{M}\cong \faktor{\mathcal{Z}}{(\pi^r) \mathcal{Z}}\cong \Big(\faktor{R_\mathcal{M}}{(\pi^r) R_\mathcal{M}} \Big)^{\oplus s} \cong
\mathrm{Hom}_{\mathcal{O}_\mathbb{K}}\Big(\faktor{R_\mathcal{M}}{(\pi^r)R_\mathcal{M}},\faktor{\mathcal{O}_\mathbb{K}}{(\pi^r)}\Big)^{\oplus s},$$
where the last isomorphism of $R_\mathcal{M}$-modules holds because of the extension of scalars property for homomorphism modules (see Prop. 10, chpt. 1, \cite{Bou89}). 
The inclusion of ideals $(\pi^r) \subseteq I$ and the exactness of the localization functor induce the isomorphism of $R_\mathcal{M}$-modules:

$$W[I]_\mathcal{M}\cong \mathrm{Hom}_{\mathcal{O}_\mathbb{K}}\Big(\faktor{R_\mathcal{M}}{(\pi^r)R_\mathcal{M}},\faktor{\mathcal{O}_\mathbb{K}}{(\pi^r)}\Big)[I]^{\oplus s}. $$
 Moreover, the action of $R_\mathcal{M}$ on $W[I]_\mathcal{M}$ factors through the quotient $R_\mathcal{M}/{I R_\mathcal{M}}$, and so the above isomorphism is an isomorphism of $R_\mathcal{M}/{I R_\mathcal{M}}$-modules.\\
 We observe that, since $\pi^r \in I$, there exists an ideal $\bar{I}$ of $R/{(\pi^r)R}$ such that: $$\faktor{\faktor{R}{(\pi^r)}}{\bar{I} }\cong \faktor{R}{I}.$$
 The $\mathcal{O}_\mathbb{K}$-algebra homomorphism $h$ is surjective and it factors through the quotient by its kernel, so $R/I$ is isomorphic (as a ring) to $\mathcal{O}_{\mathbb{K}}/(\pi^r)$ which is a local ring, hence $R/I$ is a local ring as well. By the exactness of the localization functor we have a ring isomorphism $R_\mathcal{M}/{I R_\mathcal{M}}\cong R/I$.\\
It follows that we have an isomorphism of $R/I$-modules:
 
 $$
 \begin{aligned}
 W[I]\cong W[I]_\mathcal{M}&\cong \mathrm{Hom}_{\mathcal{O}_\mathbb{K}}\Bigg(\faktor{\Big(\faktor{R_\mathcal{M}}{(\pi^r)R_\mathcal{M}}\Big)}{\bar{I}R_\mathcal{M}},\faktor{\mathcal{O}_\mathbb{K}}{(\pi^r)}\Bigg)^{\oplus s}\\&\cong \mathrm{Hom}_{\mathcal{O}_\mathbb{K}}\Big( \faktor{R_{\mathcal{M}}}{I R_{\mathcal{M}}},\faktor{\mathcal{O}_\mathbb{K}}{(\pi^r)}\Big)^{\oplus s}\\&\cong \mathrm{Hom}_{\mathcal{O}_\mathbb{K}}\Big( \faktor{R}{I},\faktor{\mathcal{O}_\mathbb{K}}{(\pi^r)}\Big)^{\oplus s},
 \end{aligned}
 $$
 so $W[I]$ is a faithful $R/I$-module and the proof is complete.

\end{proof}
Now, we want to identify a cuspidal eigenform modulo $\pi^r$ of generic level $N$ (not divisible by $p$) with a unique class of Hecke submodules of $W[\pi^r]$. In order to do so we will make use of lemma \ref{bello}. We will again state the result in the context of commutative algebra. We keep the same setting as the previous lemma. If $J$ is an ideal of $R$ which is contained in a unique maximal ideal, we will denote such maximal ideal as $\mathcal{M}_J$. We define the set:
$$\mathfrak{B}:=\Big\{M\subseteq W[\pi^r] \;\;R \text{-submodule}:\;  \faktor{R}{\text{Ann}_R (M)}\cong \faktor{\mathcal{O}_\mathbb{K}}{(\pi^r)} \text{ as } \mathcal{O}_\mathbb{K}\text{-algebras}; \; R_{\mathcal{M}_{\text{Ann}_R (M)}} \text{ is Gorenstein}\Big\}.$$
Let $\sim$ be the equivalence relation on $\mathfrak{B}$ given by: $M\sim N$ if and only if $\text{Ann}_R (N)=\text{Ann}_R (M)$.
We have the following lemma:

\begin{lemma} \label{bah} Define the map between sets as follows:
$$
\begin{aligned}
\varphi: \mathfrak{A}:=\Big\{h\in\text{Hom}_{\mathcal{O}_\mathbb{K}\text{-Alg}}\Big(R, \faktor{\mathcal{O}_\mathbb{K}}{(\pi^r)}\Big): R_{\mathcal{M}_{\text{Ker}(h)}}& \text{ is Gorenstein}\Big\}\longrightarrow \faktor{\mathfrak{B}}{\sim}\\
h&\longmapsto (W[\text{Ker}(h)])_{\sim}
\end{aligned}
$$
where the symbol $(\cdot)_{\sim}$ denotes the class under the equivalence $\sim$. \\
Then $\varphi$ is a bijection. 
\end{lemma}
\begin{proof}
First we prove the injectivity of $\varphi$. Take $h_1$ and $h_2$ in $\mathfrak{A}$ and assume that $\varphi(h_1)=\varphi(h_2)$. Then by definition we have that $(W[\text{Ker}(h_1)])_\sim=(W[\text{Ker}(h_2)])_\sim$ and so $\text{Ann}_R (W[\text{Ker}(h_1)])=\text{Ann}_R (W[\text{Ker}(h_2)])$. By lemma \ref{propspecial}, we deduce that $\text{Ker}(h_1)=\text{Ker}(h_2)$ and hence the injectivity follows. \\
For the surjectivity, let $(M)_\sim \in \mathfrak{B}$ and denote by $J_M$ the annihilator $\text{Ann}_R (M)$. Then the natural projection $h_M : R \rightarrow R/J_M$ satisfies $h_M\in\mathfrak{A}$ and $\varphi(h_M)=(W[J_M])_\sim$. Again by lemma \ref{propspecial}, $\text{Ann}_R (W[J_M])=J_M$ and so $(W[J_M])_\sim=(M)_\sim$. This completes the proof.
\end{proof}
\begin{remark} \normalfont
Finally, as in remark \ref{bello}, we can apply the above lemma to $H=W(\mathcal{O}_\mathbb{K})$, and $R=\mathbb{T}_N$ and have a correspondence between cuspidal eigenforms modulo prime powers and certain classes of Hecke modules. 
\end{remark}

\section{Examples}

\pagestyle{plain}

All the following computations are made using MAGMA (see \cite{MAGMA}) and the LMFDB database (see \cite{lmfdb}).\\ We will keep the same notation as before, i.e. $k\geq 2$ and $N\geq 5$ are integers, $p$ is an odd prime and $\mathbb{K}$ is a sufficiently big finite extension of $\mathbb{Q}_p$.\\
In this section we will present examples related to level raising, underlining the difference between full level raising and partial level raising and connecting Theorem \ref{MainTheorem} 
and Diamond's Theorem \ref{Diamond}, which we now state completely:

\begin{theorem}
\label{Diamond}
Let $f$ be a newform of level $N$, weight $k$, character $\chi$ and coefficients in $\mathbb{K}$ satisfying the level raising condition $(\text{LRC})^2$ 
at a prime $l$ modulo $\pi^r$, i.e. $a_l(f)^2-(l+1)^2 l^{k-2}\chi (l)\equiv 0 \text{ mod }\pi^r$. Assume that $p\nmid lN \varphi (N)(k-2)!$. \\
Then there exists a finite family of positive integers $d_i$ and distinct $l$-newforms $g_i$ of level $lN$ (i.e. with respect to the congruence subgroup $\Gamma_1(N) \cap \Gamma_0 (l))$ such that:

$$
\begin{aligned}
(i)&\;\; f \equiv g_i \text{ mod }\pi^{d_i}   \\
(ii)&\;\; \sum_i d_i \geq r
\end{aligned}
$$
\end{theorem}
Let $f$ be a cuspidal eigenform satisfying the hypothesis of the above theorem. We will say that $f$ satisfies the full level raising modulo $\pi^r$ at a prime $l$ if, applying Theorem \ref{Diamond}, we have that there exists an index $i$ and a $l$-newform $g_i$ such that $f\equiv g_i \text{ mod }\pi^r$, i.e. $d_i=r$. 
\subsection{Full level raising modulo prime powers}

Consider the complex vector space $S_4(\Gamma_0(22))^{\text{new}}$, it has dimension 3. Let $f$ be the newform whose $q$-expansion is $f(q)=q-2q^2-7q^3+4q^4-19q^5+14q^6 \dots$, it has $\mathbb{Q}$ as coefficients field. \\
For the primes $l=5$ and $p=7$, the form $f$ satisfies the following level raising conditions:

$$
\begin{aligned}
&(\text{LRC})^2: \;\; a_5(f)^2-(5+1)^2 5^2 \equiv 0 \text{ mod } 7^2,\\
&(\text{LRC})_{+}: \;\; a_5(f)-(5+1) 5 \equiv 0 \text{ mod } 7^2,\\
&(\text{LRC})_{-}: \;\; a_5(f)-(5+1) 5 \not\equiv 0 \text{ mod } 7,\\
\end{aligned}
$$
indeed note that, by Lemma \ref{confronto}, these conditions are equivalent since $p$ does not divide $l+1$.\\
In accordance with Diamond's result, we find a newform  $g\in S_4 (\Gamma_0 (110))^{\text{new}}$ such that $f$ is congruent to $g$ modulo $7^2$, i.e. $a_q (f)\equiv a_q (g) \text{ mod }7^2$ for all primes $q\not=2, 5, 7,11$.\\ 
Now, denote by $\mathbb{T}_{110}$ the Hecke algebra of level $\Gamma_0(110)$. As predicted by our result (see Thm.\ref{MainTheorem}) there is a cuspidal eigenform modulo $7^2$ of level $110$, say $h:\mathbb{T}_{110} \rightarrow \mathbb{Z}_7/7^2\mathbb{Z}_7$, such that $h$ as $\mathbb{Z}_7$-algebra homomorphism factors through the $\mathbb{Z}_7$-algebra $\mathbb{T}_{110}^\text{7-new}$ (because in particular, it factors through $\mathbb{T}_{110}^{\text{new}}$). As ring homomorphism, the eigenform $h$ lifts in characteristic 0, in particular it is nothing else than the reduction of the classical newform $g$ modulo $7^2$. \\
Here is a list of other examples of full level raising, everything is in accordance with Theorem \ref{MainTheorem} and Theorem \ref{Diamond}:\\
\newline
(1) Let $f(q)=q-2q^2-5q^3-4q^4+ \dots$ be the unique newform in $S_4(\Gamma_0(23))^{\text{new}}$ with rational coefficients. Then $f$ satisfies the level raising condition $(\text{LRC})^2$ and $(\text{LRC})_{-}$ modulo $5^3$ at $l=13$.\\
\newline
(2) Let $f(q)=q-5q^2-7q^3+17q^4-7q^5+\dots$ be the unique newform in $S_4(\Gamma_0(13))^{\text{new}}$ with rational coefficients. Let $l=23$ and $p=3$, then $f$ satisfies at $l$ the conditions $(\text{LRC})^2$ modulo $3^5$, $(\text{LRC})_{-}$ modulo $3$ and $(\text{LRC})_{+}$ modulo $3^4$. Note that the conditions are not equivalent since $v_p(l+1)=1$. It is an example of full level raising not with respect of the condition $(\text{LRC}^2)$, but with respect to $(\text{LRC}_{+})$. There is indeed a congruence modulo $3^4$ between $f$ and a $23$-newform of level $299$. In particular, this example shows why the conditions $(\text{LRC})_\epsilon$ for $\epsilon\in\{\pm1\}$ are the level raising conditions that need to be considered in the prime powers case. 
\subsection{Partial level raising modulo prime powers}
As before, consider the 3-dimensional complex vector space $S_4(\Gamma_0 (22))^{\text{new}}$. Let $f\in S_4(\Gamma_0 (22))^{\text{new}}$ be the newform whose $q$-expansion is $f(q)=q+2q^2+q^3+4q^4-3q^5+2q^6 \dots$, it has $\mathbb{Q}$ as coefficient field. \\
It turns out that for the primes $l=5$ and $p=3$ the form $f$ satisfies the following non-equivalent level raising conditions:
$$
\begin{aligned}
&(\text{LRC})^2: \;\; a_5(f)^2-(5+1)^2 5^2 \equiv 0 \text{ mod } 3^4,\\
&(\text{LRC})_{+}: \;\; a_5(f)-(5+1) 5 \equiv 0 \text{ mod } 3,\\
&(\text{LRC})_{-}: \;\; a_5(f)+(5+1) 5 \equiv 0 \text{ mod } 3^3.\\
\end{aligned}
$$
In accordance with Diamond's theorem (see Thm.2 in \cite{Diamond1991}), we find that there exists a family of $5$-newforms $\{g_i: \; i=1,2,3,4,5\}$ such that $g_i \equiv f \text{ mod }3$ if $i\not=5$ and $g_5\equiv f \text{ mod }3^2$. More precisely, $g_1, \; g_2, \; g_3, \; g_4$ are newforms of level $110$ and $g_5$ is a newform of level $10$:

$$
\begin{aligned}
&g_1(q)=q+2q^2+q^3+4q^4+5q^5+2q^6+23q^7+8q^8\dots,\;\;      &g_1\equiv f \text{ mod } 3   \\ 
&g_2(q)=q+2q^2+7q^3+4q^4-5q^5+14q^6+1q^7+8q^8\dots,\;\;      &g_2\equiv f \text{ mod }  3    \\ 
&g_3(q)=q+2q^2-8q^3+4q^4-5q^5-16q^6+26q^7+8q^8\dots,\;\;  \;  \;  &g_3\equiv f \text{ mod }  3^2    \\ 
&g_4(q)=q-2q^2+4q^3+4q^4+5q^5-8q^6+20q^7-8q^8\dots,\;\;       &g_4\equiv f \text{ mod }   3    \\ 
&g_5(q)=q+2q^2-8q^3+4q^4+5q^5-16q^6-4q^7+8q^8\dots,\;\;       &g_5\equiv f \text{ mod }   3\\
\end{aligned}
$$
This gives an example of a strict inequality in the relation $\sum_i d_i \geq r$ where in this case $r=4$ and $d_i=1$ for $i\not=5$ and $d_5=2$. Note that there are no congruences modulo $3^4$ as predicted.\\
Concerning the more general case of cuspidal eigenforms modulo prime powers, everything behaves according to our theorem \ref{MainTheorem}. In particular, the level raising condition $(\text{LRC})_{-}$, which holds modulo $3$, gives rise to a cuspidal eigenform modulo $3$ which is $5$-new, we call it $h_1: \mathbb{T}_{110}\rightarrow \mathbb{Z}_3/3\mathbb{Z}_3$ such that $f\equiv h_1 \text{ mod }3$ and the level raising condition $(\text{LRC})_{+}$, which holds modulo $3^3$, gives rise to a cuspidal eigenform modulo $3^3$ which is $5$-new, say $h_2: \mathbb{T}_{110}\rightarrow \mathbb{Z}_3/3^3\mathbb{Z}_3$ such that $f\equiv h_2 \text{ mod }3^3$.\\
 Note that, as predicted by the Deligne-Serre's lifting lemma, the $h_1$ lifts in characteristic 0 as $\mathbb{Z}_3$-algebra homomorphism while $h_2$ does not since there are no congruence modulo $3^3$ between $f$ and classical $5$-newforms of level $110$. Hence, this gives a counterexample for the full level raising even when the level raising condition $(\text{LRC})^2$ is replaced by the condition $(\text{LRC})_\epsilon$ for some $\epsilon\in\{\pm 1\}$. Nevertheless, it is interesting to observe that in this case the cuspidal eigenform modulo $3^3$ that we called $h_2$ and which does not lift in characteristic 0 (as a ring homomorphism) can be recovered by a linear combination of the classical $5$-newforms of Diamond's theorem. More specifically, we have that $h_2\equiv f+3g_2+5g_3+18g_4 \text{ mod }3^3$. The cusp form $f+3g_2+5g_3+18g_4$ is not an eigenform in characteristic zero but it becomes an eigenform when reduced modulo $3^3$. \\
Here is another example of partial level raising, everything is in accordance with Theorem \ref{MainTheorem} and Theorem \ref{Diamond}:\\
\newline
Let $f(q)=q+2q^2+4q^4-6q^5-16q^7+ \dots$ be the unique newform in $S_4(\Gamma_0(18))^{\text{new}}$. Then $f$ satisfies the following non-equivalent level raising conditions for $p=5$ at $l=29$:
$$
\begin{aligned}
&(\text{LRC})^2: \;\; a_{29}(f)^2-(29+1)^2 29^2 \equiv 0 \text{ mod } 5^3,\\
&(\text{LRC})_{+}: \;\; a_{29}(f)-(29+1) 29 \equiv 0 \text{ mod } 5^2,\\
&(\text{LRC})_{-}: \;\; a_{29}(f)+(29+1) 29 \equiv 0 \text{ mod } 5.\\
\end{aligned}
$$
It is an example of partial level raising because there are no congruences  modulo $5^2$ between $f$ and $29$-newforms of level $522$.

\printbibliography[heading=bibintoc,title={References}]

\end{document}